\documentclass{article}

\usepackage[table]{xcolor}
\usepackage{geometry}
\usepackage{amsthm,amssymb,amsfonts,amsmath}
\usepackage{graphicx}
\usepackage{url}

\theoremstyle{theorem}
\newtheorem{theorem}{Theorem}
 \newtheorem{lemma}[theorem]{Lemma}
\newtheorem{proposition}[theorem]{Proposition}

\theoremstyle{definition}
\newtheorem*{definition}{Definition}

\begin{document}

\title{Hypercups: Flipping Cups With More Than Two Sides}
\author{Micah Dykhuis and Lauren Keough and Sydney Lipton } 

\date{}

\maketitle

    \begin{abstract}
  In their 2010 article entitled ``How to invert $n$ cups $m$ at a time" in {\it Mathematics Today}, Man-Keung Siu and Ian Stewart extend a classic trick in which $3$ cups are flipped $2$ at a time to any number of cups, flipped any number at a time.  Here, we generalize another feature of the problem. A typical cup in our universe, as far as we know, has two ``sides", right side up and upside down. But what if the cup had $k$ ``sides" with $k\ge 2$? We call these $k-$hypercups. The classic trick depends on parity. In this article, we show the analogous hypercups trick depends on greatest common divisors. We generalize all of Siu and Stewart's results to $k-$hypercups. Our results imply the known results for $2-$hypercups, also known as cups.
  \end{abstract}

A mathematician has many powers - the powers of deduction, generalization, and problem solving, to name a few. But as Uncle Ben in Spider Man said, ``with great power comes great responsibility". For example, a mathematician must always show their work, and, as we must warn you before we tell you a trick, a mathematician must never use their mathematical knowledge to trick innocent people.

Perhaps you have seen this trick on social media. For example, one such video is called ``THE IMPOSSIBLE 3 CUP TRICK!! WHY CAN'T I DO IT??  \#wow \#magic \#cool \#impossible" \cite{facebook}. To start, you need 3 cups (or use our Desmos model at \url{https://tinyurl.com/3cupstrick}). You begin by arranging the cups as in Figure~\ref{fig:winningcombo}. 
\begin{figure}[h!]
    \centering
    \includegraphics[width=0.5\linewidth]{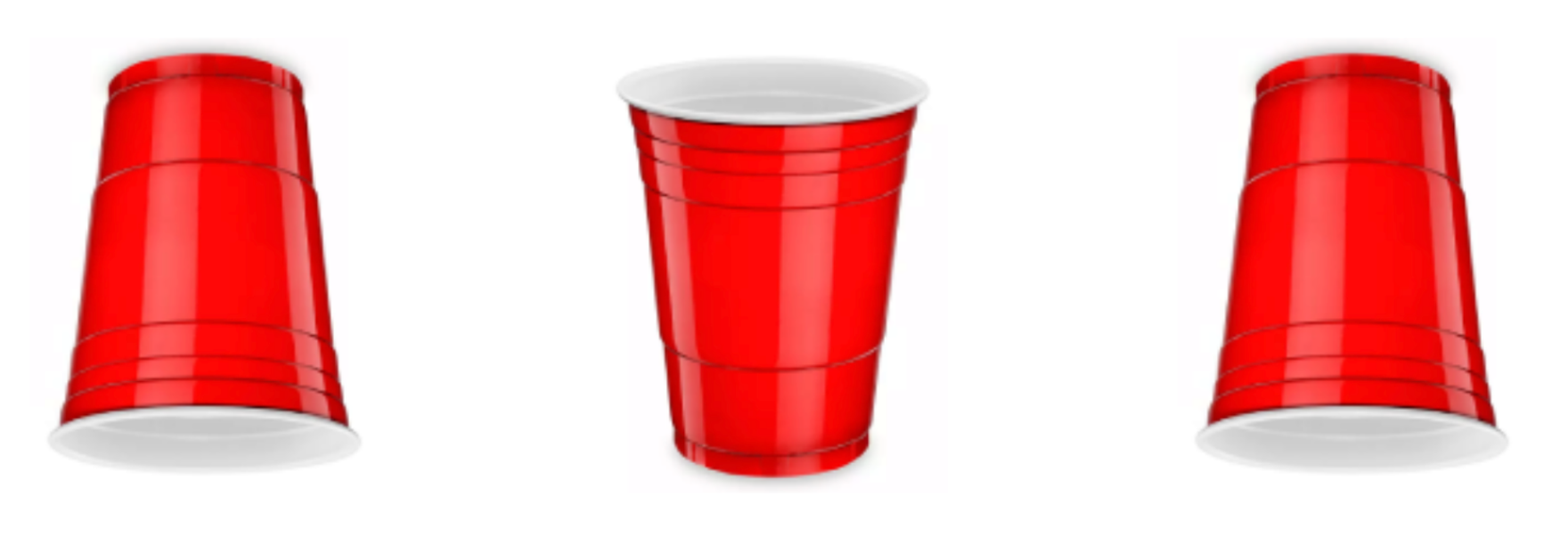}
    \caption{The initial configuration you use in the 3 cup trick.}
    \textit{Figure 1 alt text: Three red solo cups; the first and last are upside down and the center is right side up.}
    \label{fig:winningcombo}
\end{figure}

 Now you can perform the trick on the target. Show them that you can flip 2 cups at a time to get the cups all right side up. Of course, one can do this in one move, but for the sake of trickery, do it in 3 moves, say flipping the first two, then the outside two, and then the first two again. No problem.

Now, from the finishing configuration (all cups right side up), flip the middle cup to down, as in Figure~\ref{fig:losingcombo} and ask your target to flip all the cups to right side up by flipping 2 at a time.\footnote{The second author tried this trick with a middle school class early in the morning, figuring that middle schoolers would not be very observant before 8AM. She was wrong. They quickly pointed out the second initial configuration was not the same as the first.}

\begin{figure}[h!]
    \centering
    \includegraphics[width=0.5\linewidth]{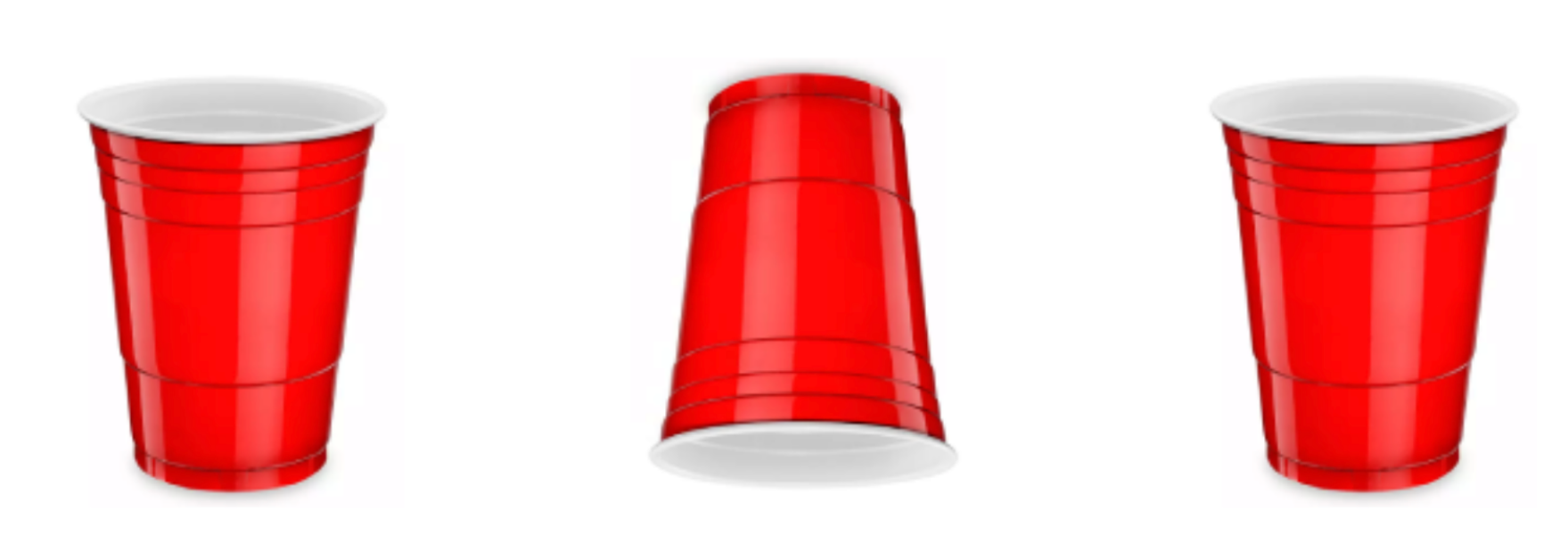}
    \caption{The configuration the target starts with in the 3 cup trick.}
    \textit{Figure 2 alt text: Three red solo cups; the first and last are right side up and the center is upside down.}
    \label{fig:losingcombo}
\end{figure}

If you try it for yourself, we guarantee that you will not succeed. Why? Because math! More specifically, from any configuration, if you flip two cups, one of the following three scenarios will occur  
\begin{enumerate}
\renewcommand{\theenumi}{(\alph{enumi})}
    \item flip two right side up cups to upside down. 
    \item two upside down cups to right side up. 
    \item one right side up cup to down and one upside down cup to right side up. 
\end{enumerate}

Let's consider the effect of each of these scenarios on the net change of the number of right side up cups. In scenario (a) the net change is $+2$ right side up cups, in scenario (b) the net change is $-2$, and in scenario (c), the net change is $+0$. In the initial configuration in Figure~\ref{fig:losingcombo}, the number of right side up cups is $2$, and one needs that number to be $3$ to finish the task. But one cannot get from $2$ to $3$ using a combination of $+2, -2,$ and $+0$. This is an issue of parity, as $2$ is even and $3$ is odd.

Using mathematicians' power responsibly, Man-Keung Siu and Ian Stewart generalize the problem (not the trick) to any number of cups, flipped any number at a time \cite{SS10}. In their case, start with $n$ cups, all right side up. When can one, flipping precisely $m$ cups at a time, get to all cups upside down, and what is the minimum number of moves to do so? Siu and Stewart show that whether a solution exists depends on the parity of $m$ and $n$. More specifically, a solution exists if and only if it $n$ is even or $m$ is odd. In other words, the only situation where a solution does not exist is if $n$ is odd and $m$ is even simultaneously. Moreover, they provide the number of moves depending on the parity of $m$ and $n$, and on the comparison of $2m$ and $n$, see Table~\ref{tab:minmoves2cups}.

\begin{table}[h!]
\centering
\scalebox{0.8}{
\begin{tabular}{ |m{2.5cm}|m{2.5cm}|m{2.5cm}|m{2.5cm}|m{2.5cm}|  }
 \hline
\rowcolor{black} \multicolumn{5}{|c|}{\textcolor{white}{Minimum number of moves for inversion. Here $\lceil x\rceil$ is the}} \\
\rowcolor{black} \multicolumn{5}{|c|}{\textcolor{white}{  \textit{ceiling function}, the smallest integer $\geq x$.}} \\
\hline
 & $n   $ even, $ m$ even & $n $ odd, $m$ odd & $n$ odd, $m$ even & $n$ even, $m$ odd\\
 \hline
 $2m \leq n$\vspace{0.25in}   & $\lceil \frac{n}{m} \rceil$    & $2 \lceil \frac{n-m}{2m}\rceil +1$ &  no solution & $2\lceil \frac{n}{2m} \rceil$\\
 \hline
 $2m > n$ \vspace{0.1in}&   $1$ if $m = n$  &  $1$ if $m = n$    & no solution & $2\lceil \frac{n}{2(n-m)} \rceil$\\
 & $3$ if $m > n$ & $3$ if $m > n$ & &\\
 \hline
\end{tabular}}
\caption{Minimum number of moves for $2$-hypercups from \cite{SS10}}
\label{tab:minmoves2cups}
\end{table}

Here, we generalize the problem again, to $k-$hypercups, an idea first introduced at a talk \cite{SSDTalk}. For $k-$hypercups, instead of right side up and upside down, we have states $0,1,\dots, k-1$ for $k\ge 2$. Instead of flipped, we say hypercups are rotated.

\begin{definition}\label{def:hypercup}
    A \textit{$k-$hypercup} is a mathematical object which has possible states $0,1,2,\dots, k-1$. When a hypercup is rotated, its state increases by $1$ modulo $k$. In particular, if a hypercup in state $k-1$ is rotated, the hypercup returns to state $0$.
\end{definition}

For the remainder of the paper we assume $n,m,k\in\mathbb{N}$, with $m\le n$, and $k\ge 2$. For the $(n,m,k)$ hypercup problem we start with $n$ $k-$hypercups in state $0$, and rotate $m$ hypercups at a time. Our goal is to end with all hypercups in state $k-1$. 

\begin{definition}\label{def:solved}
    The $(n,m,k)$ hypercup problem is \textit{solvable} when the $n$ $k-$hypercups, all starting in state $0$, can be rotated $m$ at a time to result in all hypercups having state $k-1$. Similarly, the $(n,m,k)$ hypercup problem is \textit{solved} when $s=n(k-1)$, where $s$ is the sum of the states of all $n$ hypercups.
\end{definition}

You can try the $(n,m,k)$ hypercup problem at \url{https://tinyurl.com/HypercupsTrick}. Choose different numbers for $n$, $m$, and $k$, and then starting with all $n$ hypercups at state $0$, and rotating $m$ at a time, get all hypercups in state $k-1$. For what $n,m,k$ can you complete the task? Spoiler alert - the result is in the next theorem!

To begin exploring the problem, we used the NetworkX Python Library \cite{NetworkX} to generate directed graphs to find solutions to the hypercup problem. Each node represented a possible set of states for the $n$ $k$-hypercups, and an arc connecting two nodes represented the move that changes the set of states of the former node to the latter. We then used a Breadth First Search to generate data about the solvability and the minimum number of moves for $2\le k\le 10$ and $m\le n$ for varying maximum values of $n$ due to a computational time complexity of $\mathcal{O}\left(\binom{n+k-1}{n} \cdot \binom{n}{m} \cdot n \log n\right)$. Like Siu and Stewart \cite{SS10}, we saw very few obvious patterns in the data we generated, so we were quite surprised by the following result!




\begin{theorem}\label{thm:bigiff}
    The $(n,m,k)$ hypercup problem is solvable if and only if $\gcd(m,k)\mid n$.
\end{theorem}

Stay tuned to see the proof of Theorem~\ref{thm:bigiff}. For now, consider the $2-$hypercups case. Note $\gcd(m,2)\in\{1,2\}$, and $\gcd(m,2) = 2$ if and only if $m$ is even. Additionally $2\mid n$ if and only if $n$ is even. So there is no solution exactly when $n$ is odd and $m$ is even. Thus the result on solvability of the $2-$hypercup problem in \cite{SS10} follows from Theorem~\ref{thm:bigiff}.

Moreover, we find the minimum number of moves required. Let $M(n,m,k)$ be the least number of moves require to solve the $(n,m,k)$ hypercup problem. 

\begin{theorem}\label{thm:numberofmoves}
    For all $n,m,k\in\mathbb{N}$ with $k\ge 2$ and $n\ge m$, if the $(n,m,k)$ hypercup problem is solvable, then the minimum number of moves required is
        \[ M(n,m,k) = \frac{n(k-1)+j_\sigma k}{m}\]
    where $j_\sigma$ is the least nonnegative integer such that 
    \begin{enumerate}
        \item \textbf{Divisibility Condition}: $m\mid (n(k-1)+j_\sigma k)$, and
        \item \textbf{Reset Condition}: $\frac{n(k-1)+j_\sigma k}{m}\ge (k-1) + k\left\lceil\frac{j_\sigma}{n}\right\rceil$.
    \end{enumerate} 
\end{theorem}

Good  news - the formula in Theorem~\ref{thm:numberofmoves} exactly matches Siu and Stewart's results in Table~\ref{tab:minmoves2cups} for $2-$hypercups. The proof of that is by case analysis and we leave it as an exercise to the reader\footnote{We know we said we should always show our work, but this work is not particularly insightful.}. We admit the formula looks a bit complicated, but on a more positive note - the formula for the $2$-hypercups requires $6$ cases (see Table~\ref{tab:minmoves2cups}) and here we have one formula that also generalizes to $k-$hypercups for any $k\ge 2$. Additionally, we'll show that what $j_\sigma$ represents and the two given conditions can be explained naturally in relation to the problem. 

\section*{The Necessity of $\gcd(m,k)\mid n$}\label{sec:forward}

Thanks for staying tuned for the proof of Theorem~\ref{thm:bigiff}. In this section we prove the forwards direction, that is if the $(n,m,k)$ hypercup problem has a solution then $\gcd(m,k)\mid n$. First we state and prove a lemma about the sum of the states of the hypercups. Recall all hypercups start in state $0$, contributing $0$ to the sum of the states. When a hypercup is rotated the sum increases by $1$ unless the hypercup is in state $k-1$, which will then subtract $k-1$ from the sum.

\begin{lemma}\label{lem:movesum}
    In the $(n,m,k)$ hypercup problem, any move adds $m-jk$ to the sum of the states of the hypercups, for some $j\in \mathbb{Z}$, $0\le j\le m$. At any move, say $r$, the sum of the states of the hypercups is then $\sum_{i=1}^r m-j_ik$ for some $j_i$ with $0\le j_i \le m$ for $1\le i \le r$. Moreover, $j_i$ corresponds to the number of hypercups reset to $0$ on move $i$.
\end{lemma}

\begin{proof}
     On each move, we rotate $m$ hypercups at once. Say there are $0\le j \le m$ hypercups that are be in state $k-1$. Those $j$ hypercups reset to state $0$, while the other $m-j$ hypercups increase by $1$. Thus, the total sum decreases by $j(k-1)$ and increases by $1\cdot(m-j)$. Therefore, the total sum changes by 
        \begin{align*}
            (m-j)-j(k-1) = m-jk
        \end{align*}
    for some $0\le j\le m$. Since the sum of the states of the hypercups is initially $0$, and each move adds $m-jk$ to the sum of the states, the rest follows immediately.
        \end{proof}

The following proposition proves the forward direction of Theorem~\ref{thm:bigiff}.

\begin{proposition}
   If the $(n,m,k)$ hypercup problem is solvable then $\gcd(m,k)\mid n$. 
\end{proposition}
\begin{proof}
  Assume there is a solution. Then the final sum of all $n$ hypercups, which we denote as $s$, will be equal to $n(k-1)$.

  By Lemma~\ref{lem:movesum}, we also know there exist nonnegative integers $j_i$ for $1\le i\le r$ such that 
  \begin{align*}
      s=\sum_{i=1}^r m-j_ik.
  \end{align*}
 
  If we reduce the sum modulo $k$, we obtain
  \begin{align*}
      \sum_{i=1}^r m-j_ik &\equiv \sum_{i=1}^r m\pmod{k}\\ 
      &\equiv mr\pmod k.
  \end{align*}
If we reduce $n(k-1)$ modulo $k$, we obtain
$$n(k-1)\equiv -n\pmod k.$$
Since $\sum_{i=1}^r m-j_ik=s = n(k-1)$, we see that $-n\equiv mr \pmod k$. Thus $k~\mid~(-n -mr)$ and so for some integer $y$, 
    \begin{align*}
        ky &= -n-mr\\
        n&=k(-y)+m(-r).
    \end{align*}
  Note that $\gcd(m,k)$ divides $m$ and $k$, and hence any integer linear combination of $m$ and $k$. So $\gcd (m,k) \mid n$.
\end{proof}

\section*{Minimum Number of Moves: The Origin of $j_\sigma$}

Before proving the backwards direction of Theorem~\ref{thm:bigiff}, we will discuss another question answered in \cite{SS10} for $2-$hypercups: What is the minimum number of moves to reach a solution? Notably, in the original cup trick, in order to add an element of confusion (or, some might say, \#magic)  one does not do the minimum number of moves. In this section we work to motivate the formula given in Theorem~\ref{thm:numberofmoves}. 

Here's one simple situation: say $m\mid n$. If the number of hypercups you rotate at a time, $m$, is a divisor of the total number of hypercups, $n$, then the minimum number of moves must be $\frac{n(k-1)}{m}$. There is a simple algorithm we can use to show this:
    \begin{enumerate}
        \item Split the hypercups into equal groups of size $m$.
        \item For each group, rotate the hypercups in only that group $k-1$ times.
    \end{enumerate}
Note we have $\frac{n}{m}$ groups and each group gets $k-1$ moves. Thus, there are $\frac{n(k-1)}{m}$ moves.  Of course, we can't have fewer moves, since we need to reach sum $n(k-1)$, and the sum increases by at most $m$ each time. Huzzah! 

The minimum number of moves being  $\frac{n(k-1)}{m}$ would be great. But sadly, this is not true. In \cite{SS10}, the authors produce a table of the minimum number of moves for $1\leq m \leq n \leq 20$ for $2-$hypercups. They note ``there are no obvious patterns". Surely the authors would have seen that the answer was $\frac{n(2-1)}{m} = \frac{n}{m}$ if that were truly the answer. The first problem is that $\frac{n(k-1)}{m}$ doesn't even need to be an integer. For example, what if $n=12$, $m=5$, and $k=2$? Thus, sadly, the minimum number of moves can't be so simple.

Recall we are trying to get the sum of the states of the hypercups to be $n(k-1)$. But we don't need to get that sum by only ever adding $m$. Sometimes, we subtract $(k-1)$ by rotating a hypercup back to state $0$. In fact, recall Lemma~\ref{lem:movesum} which says each move adds $m-jk$ for some $0\le j \le m$. By the proof of Lemma~\ref{lem:movesum}, $j$ represents the number of hypercups rotated back to $0$ on the given move. 

Thus, for the $(n,m,k)$ hypercup problem to be solvable we need, for some number of moves $r$, there exist integers $j_1, j_2, \dots, j_r$, between $0$ and $m$, such that 
\begin{align}
   n(k-1) = \sum_{i=1}^r m-j_i k = rm - k \sum_{i=1}^r j_i.\label{eqn:desiredr}
\end{align}
Define $j_\sigma = \sum_{i=1}^r j_i$. Solving equation \ref{eqn:desiredr} for $r$ we get 
\begin{align}
    r = \dfrac{n(k-1)+j_\sigma k}{m}. \label{eqn:ris}
\end{align}
This means if we can find a nonnegative integer $j_\sigma$ that makes $r$ an integer, then $r$ could potentially be a number of moves that solves the $(n,m,k)$ hypercup problem! In fact, consider $n=12, m=5$ and $k=2$ again. Choosing $j_\sigma = 4$ we find the expression on the right of equation~\ref{eqn:ris} is
\begin{align*}
    \frac{12(2-1)+4\cdot 2}{5} = 4,
\end{align*}
which is not only an integer, but the minimum number of moves required for the $(12,5,2)$ hypercup problem \cite{SS10}. This motivates the structure of the formula and the Divisibility Condition in Theorem~\ref{thm:numberofmoves}. Moreover, we can interpret $j_\sigma$ the total number of times there is a hypercup that goes from state $k-1$ to state $0$. Woohoo!

At this point in the research journey, one might make the seemingly reasonable conjecture that the minimum number of moves is $r = \frac{n(k-1)+j_\sigma k}{m}$, where $j_\sigma$ is the smallest nonnegative integer such that  $r\in \mathbb{Z}$. Unfortunately, one would be incorrect. Consider the $(8,6,2)$ hypercup problem. Then, choosing $j_\sigma = 2$, we get 
\begin{align}
    \dfrac{n(k-1)+j_\sigma k}{m} = \dfrac{8(2-1)+2\cdot 2}{6} = 2.\label{eqn:wrong}
\end{align}
Since this is an integer, we could potentially solve this problem in $2$ moves.  But according to \cite{SS10}, the minimum number of moves is $3$, so something is wrong. The problem here is that, by the interpretation of $j_\sigma$, we need $2$ hypercups to return to state $0$ in $2$ moves. No hypercups can \textit{return} to state $0$ in the first move since they all start in state $0$. But then $2$ hypercups return to $0$ on the second move (the only other move), and aren't in state $1$, which they should be if we've done the correct $2$ moves to finish. So the problem hasn't actually been solved! More generally, we need the number of moves to allow for hypercups to reach state $k-1$, return to state $0$, and then get back to state $k-1$ a total of $j_\sigma$ times. In some cases, we may even need one particular hypercup, call it Sisyphus, to return to $0$ multiple times. We might even have multiple Sisyphuses (Sisyphi?) (Sisphyses?)\footnote{Consider $(n,m,k)=(5,4,7)$. It turns out $j_\sigma = 10$, which means every hypercup needs to reset twice. Also we are pretty sure it is Sysiphi.}. Imagine we are kind and we distribute the number of times a hypercup needs to return to state $0$ as evenly as possible. Then any one hypercup needs to return to state $0$ at most $\left\lceil \frac{j_\sigma}{n}\right\rceil$ times. The expression 
\begin{align*}
    (k-1) + k \left\lceil \dfrac{j_\sigma}{n}\right\rceil
\end{align*}
represents the number of rotations a hypercup needs to go back to $0$ and get back to state $k-1$ a total of $\left\lceil \frac{j_\sigma}{n}\right\rceil$ times. Certainly our number of moves must exceed that number. But in the $(8,6,2)$ hypercup problem, if $j_\sigma = 2$, there is a hypercup that needs
\begin{align*}
    (k-1) +  k\left\lceil\frac{j_\sigma}{n}\right\rceil 
    = (2-1) + 2\left\lceil\frac{2}{8}\right\rceil = 3,
\end{align*}
rotations, which is more than the $2$ moves given by equation~\ref{eqn:wrong}. This analysis motivates the Reset Condition stated in Theorem~\ref{thm:numberofmoves}. 

This is all well and good, but it is not a proof of Theorem~\ref{thm:numberofmoves}. Who says there isn't some other condition we are missing? Does a $j_\sigma$ meeting the Divisibility and Reset Conditions even exist in every case? And what about Theorem~\ref{thm:bigiff}? What does this even have to do with $\gcd(m,k)$ dividing $n$? Read on!

\section*{The Thirsty Algorithm}

For a moment, let's forget what $j_\sigma$ represents and simply choose $j_\sigma$ to be some nonnegative integer\footnote{We will see later how to choose a $j_\sigma$ that satisfies the Divisibility and Reset Conditions, and how those lead to The Thirsty Algorithm terminating with all hypercups in state $k-1$.}. With this choice, we describe a greedy algorithm called The Thirsty Algorithm. For more about greedy algorithms, see Section 3.1.5 in \cite{GreedyAlg}.

In the algorithm, start by assigning each hypercup an initial ``thirst" as follows:
 \begin{enumerate}
     \item Every hypercup initially gets a thirst of $k-1$.
     \item For $j_\sigma$ hypercups add an extra $k$ to their thirst. We distribute $j_\sigma$ as evenly as possible among the hypercups.   In general, each hypercup gets $k \left\lceil \frac{j_{\sigma}}{n} \right\rceil$ or $k (\left\lceil \frac{j_{\sigma}}{n} \right\rceil- 1)$ additional thirst.
 \end{enumerate} 
 For the ``thirsty" (AKA greedy) part, on each move rotate $m$ hypercups with the highest thirst and decrease each of those hypercups' thirst by $1$. 

Let's see the algorithm in action. Consider the $(5,3,2)$ hypercup problem. For reasons we shall soon see, we require for 2 hypercups to reset once, and hence, $j_\sigma=2$. Because we are working with $2-$hypercups, we will assign each hypercup an initial thirst of $2-1=1.$ We distribute the $j_\sigma=2$ rotations back to zero as evenly as possible. This results in two hypercups rotating back to $0$ once, and three hypercups that never reset to zero. We can choose which two hypercups reset arbitrarily. Adding $k=2$ to the first two hypercups, we obtain the thirsts,
\begin{equation*}
    3 \quad 3\quad 1\quad 1\quad 1.
\end{equation*}
We will now perform the steps of the thirsty algorithm, rotating the hypercups with the highest thirst first, and decreasing thirst every rotation. Sometimes there are 2 or more hypercups of equal thirst in the highest $m$ thirsts, like in the initial states in Table~\ref{tab:thirstyalg}. In that case, it's okay to choose any set of hypercups with the highest $m$ thirsts.
\begin{table}[ht]
\centering
\begin{tabular}{c|cc}
\textbf{Move} & \textbf{Thirsts} & \textbf{States} \\
\hline
Initial & 3 3 1 1 1 & 0 0 0 0 0 \\
1 & 2 2 0 1 1 & 1 1 1 0 0 \\
2 & 1 1 0 0 1 & 0 0 1 1 0 \\
3 & 0 0 0 0 0 & 1 1 1 1 1 \\
\end{tabular}
\caption{The Thirsty Algorithm steps for the $(5,3,2)$ hypercup problem.}
\label{tab:thirstyalg}
\end{table}

We can see that after the final move, the thirst of the hypercups are all 0, and the state of all the hypercups are all $k-1=1.$ Thus, we have found a solution to solving the hypercups $(5,3,2)$.

\section*{The Sufficiency of $\gcd(m,k)\mid n$}\label{sec:backwards}

Now that we have the motivation for the minimum move formula to help our understanding, and The Thirsty Algorithm as a tool, we will prove the backwards direction of Theorem~\ref{thm:bigiff}. That is, in this section, we will prove that if $\gcd(m,k)\mid n$ then the $(n,m,k)$ hypercup problem is solvable.

Let's start with a lemma that shows, given a $j_\sigma$ that satisfies the conditions in Theorem~\ref{thm:numberofmoves}, The Thirsty Algorithm works as desired.
\begin{lemma}\label{lem:thirstyworks}
    For the $(n,m,k)$ hypercup problem, if there exists a nonnegative integer $j_\sigma$ such that the Divisibility and Reset Conditions are satisfied, The Thirsty Algorithm terminates in $\frac{n(k-1)+j_\sigma k}{m}$ moves with all hypercups in state $k-1$.

\end{lemma}

\begin{proof}
Assume that $j_\sigma$ is a nonnegative integer satisfying
    \begin{enumerate}
        \item \textbf{Divisibility Condition}: $m\mid (n(k-1)+j_\sigma k)$, and
        \item \textbf{Reset Condition}: $\frac{n(k-1)+j_\sigma k}{m}\ge (k-1) + k\left\lceil\frac{j_\sigma}{n}\right\rceil$.
    \end{enumerate} 

  We apply The Thirsty Algorithm using such a $j_\sigma$. First, some notation. Let $t_c^i$ be the thirst of a hypercup $c$ after move $i$, and let $T^i$ be the total thirst after move $i$, with $T^0$ being the initial thirst. By The Thirsty Algorithm, $T^0 = n(k-1) + j_\sigma k$, which is a multiple of $m$ by the divisibility condition. On each step of the algorithm, we subtract $m$ from the current total thirst, and so the thirst after move $i$, is $T^i = T^0 - mi$. If the algorithm runs to completion the total number of moves will be $M(m,n,k)= \frac{T^0}{m}=\frac{n(k-1)+j_\sigma k}{m}$. Define the number of moves, $M(m,n,k) = M$ for the remainder of the proof. Then $T^0 = mM$ and $T^i=mM - mi$.
  
  We claim that at any step, no hypercup has a thirst greater than the number of remaining moves.  Suppose, for the sake of a contradiction, a hypercup $c^*$ at move $i$ has greater than thirst than the number of remaining names. That is, $t_{c^*}^i > M - i$. Assume $i$ is the first move for which this occurs. By the reset condition, the initial thirst of each hypercup is at most $M$ and so $i\geq 1$. Thus, at move $i-1\ge 0$, $t_{c^*}^{i-1} \leq M-(i-1) = M-i+1$. Together with $t_{c^*}^i > M - i$, this implies that $t_{c^*}^i = M-i + 1$ and the hypercup $c^*$ was not rotated on move $i$. By The Thirsty Algorithm, there exist $m$ other hypercups with thirst at least $M - i + 1$ after move $i-1$. This gives  a total of $m+1$ hypercups (including $c^*$) with at least as much thirst as $c^*$, so 
  $$T^{i-1} \geq (m+1)(M-i+1).$$
  However, we also know that 
  $$T^{i-1} = m(M-i+1).$$
  This is a contradiction. So, the thirst every hypercup at each move is at most the number of moves remaining. 
  
  We now show, that as long as the thirst is positive, there will be a move one can make. Say, for the sake of a contradiction, after some move $j$, $T^j>0$, but only $m-1$ hypercups have remaining thirst. Since $T^j = m(M-j)$, by the pigeonhole principle, some hypercup has thirst at least $M-j+1$. But only $M-j$ moves remain, so we have a contradiction. So, as long as $T^j>0$, there will be at least $m$ hypercups with remaining thirst. Thus, we can rotate $m$ hypercups at a time so that each hypercup's total thirst will reach $0$ in $M(m,n,k) = \frac{n(k-1)+j_\sigma k}{m}$ moves.

  Since the initial thirst of a hypercup is $(k-1)+zk$, for $z\in \{\left\lceil \frac{j_{\sigma}}{n}\right\rceil, \left\lceil\frac{j_{\sigma}}{n}  \right\rceil- 1\}$, and the final thirst is $0$, then we've rotated that hypercup $(k-1)+zk$ times. Since the hypercup started in state $0$, the final state of the hypercup will be $k-1$.
\end{proof}

In the next lemma, we prove an appropriate $j_\sigma$ exists. Inside the proof is a remarkable fact - if you find a linear combination of $m$ and $k$ equaling $n(k-1)$, say $mx+ky = n(k-1)$ for $x,y\in \mathbb{Z}$, and then subtract ``enough" from the integer coefficient of $k$, the resulting integer coefficient of $k$ is $-j_\sigma$, the number of resets, and the resulting integer coefficient of $m$ is $M(m,n,k)$, the number of moves. If $\gcd(m,k)\mid n$, then there exists a linear combination of $m$ and $k$ equaling $n(k-1)$ by Bezout's Identity, which can be found using the Extended Euclidean Algorithm. So not only does $j_\sigma$ exist, we have a way to find it! For more about the Extended Euclidean Algorithm and Bezout's identity, see \cite{NumberTheory}.

\begin{lemma}\label{lem:jsigexists}
    If $\gcd(m,k)\mid n$ then there exists $j_\sigma$ satisfying the Divisibility and Reset Conditions.
\end{lemma}

\begin{proof}
    Suppose $\gcd(m,k)\mid n$. Then, by Bezout's identity there exists an integer linear combination of $m$ and $k$ equal to $\gcd(m,k)$, which can be found by using the Euclidean Algorithm. Multiplying that linear combination through by $\frac{n(k-1)}{\gcd(m,k)}\in\mathbb{Z}$ there exist integers $x$ and $y$ satisfying the linear combination
    $$mx+ky = n(k-1).$$
    In fact, there are infinitely many such pairs of integers. Given a solution $(x,y)\in\mathbb{Z}\times \mathbb{Z}$, we'll have, for any $i\in\mathbb{Z}$,
    $$m(x+ki) + k(y-mi) = n(k-1).$$
    Thus, there exists a pair $(x_0, y_0)$ where $mx_0+ky_0 = n(k-1)$ and $y_0\leq 0$. Let $j_\sigma = -y_0$. For any non-positive $y_0$, $m\mid n(k-1)-ky_0$, or $m\mid n(k-1)+ kj_\sigma$, which means the Divisibility Condition is satisfied. Additionally,
    $$x_0 = \frac{n(k-1) +kj_\sigma}{m} = M(m,n,k).$$

    So, we need only show that there is a pair $(x_0,y_0)$ with $y_0<0$ such that 
    $$x_0 \geq (k-1) + k \left\lceil \frac{-y_0}{n}\right\rceil.$$
    Continue adding $k_i$ to the coefficient of $m$ until you find $x_0\geq 0$ such that
        \begin{align}
            x_0 \geq \left( \frac{n}{n-m}\right) [(k-1)+k(n-1)-1] \label{ineq:choosex0}
        \end{align}
    and let $y_0$ be such that $mx_0+ky_0 = n(k-1)$.
    Note the right hand side is some fixed number once we choose $n,m$ and $k$. Because we can continue to add to $ki$ to the coefficient of $m$ we can reach this inequality. Plus, while we're doing that, we are subtracting from the coefficient of $k$, and so that coefficient remains negative.

  Starting with the inequality (\ref{ineq:choosex0}), we will show the Reset Condition is satisfied.   Note since $m\le n$, $\frac{n}{n-m} \geq 0$. Also, since $mx_0+ky_0 = n$, $y_0=\frac{n-mx_0}{k}$. Then,
    \begin{align*}
            x_0 &\geq \left( \frac{n}{n-m}\right) ((k-1)+k(n-1)-1)\\
            x_0 \left(\frac{n-m}{n}\right) &\geq (k-1)+k(n-1)-1\\
            x_0 &\geq (k-1) + k(n-1) +\left(\frac{mx_0}{n} - 1\right)\\
            x_0 &\geq (k-1) + k(n-1) + k \frac{\left(\frac{mx_0-n}{k}\right)}{n}\\
            x_0 &\geq (k-1) + k\left( (n-1) + \left(\frac{-y_0}{n}\right)\right)\\
            x_0 & \geq (k-1) + k \left\lceil \frac{-y_0}{n}\right\rceil.
        \end{align*}
Letting $j_\sigma = -y_0$ we see that we have the Reset Condition. Thus, there exists $x_0$ and $y_0$ satisfying both conditions.
\end{proof}

Lemmas \ref{lem:thirstyworks} and \ref{lem:jsigexists} prove that if $\gcd(m,k)\mid n$ then the $(n,m,k)$ hypercup problem is solvable.  This concludes the proof of Theorem~\ref{thm:bigiff}. Yippee!

\section*{Minimum Number of Moves: A Proof}

Finally, we'll prove Theorem~\ref{thm:numberofmoves}.

\begin{proof}
    Let $r$ be some (not necessarily minimum) number of moves that solves the $(n,m,k)$ hypercup problem. Let $S^i$ be the sum of the states of the hypercups after move $i$. Then $S^r = n(k-1)$ and, by Lemma~\ref{lem:movesum}, there exist nonnegative integers $j_i$ for $1\le i\le r$ such that $S^r = \sum_{i=1}^r m - j_i k =rm - j_\sigma k$, where $j_\sigma = \sum_{i=1}^r j_i$. Therefore, solving for $r$,
        \begin{align}
            r = \frac{n(k-1)+j_\sigma k}{m} \label{eqn:Rmoves}
        \end{align}
    for some nonnegative integer $j_\sigma$. Since $r$ is a number of moves that solves the problem, $r$ is an integer, and so the Divisibility Condition is satisfied. Moreover, $j_\sigma$ corresponds to the number of resets, so $j_\sigma$ is nonnegative. 
    
    The number of times any individual hypercup $c$ is rotated in $r$ moves is $r_c = (k-1) + k j_c$ where $j_c$ is the number of times hypercup $c$ is reset back to $0$.  One can distribute $j_\sigma$ so that any hypercup has at most $\left\lceil \frac{j_\sigma}{n}\right\rceil$ resets and so $j_c\le j_\sigma$. Clearly, you can't rotate a hypercup more times than you have moves, so $r_c \le r$ which gives the Reset Condition. 
    So, for any number of moves that solves the problem, there exists a $j_\sigma$ that satisfies equation \ref{eqn:Rmoves} and the Divisibility and Reset Conditions. Moreover, by Lemma~\ref{lem:thirstyworks}, any nonnegative $j_\sigma$ satisfying these conditions leads to a solution via The Thirsty Algorithm. 
    
    Since equation (\ref{eqn:Rmoves}) is increasing in $j_\sigma$, the minimum number of moves corresponds to the least nonnegative $j_\sigma$ satisfying the Divisibility and Reset Conditions.
\end{proof}

If the reader is interested in exploring more about hypercups, consider the following questions.  As we described, we used Python's NetworkX library to generate digraphs corresponding to the $(n,m,k)$-hypercup problem. The problems that were solvable corresponded to digraphs where the initial state and the final state were nodes in the same component. Can one determine how many components of the diagram there will be? And what nodes will be in each component? In another direction, what if you could have a heterogeneous mix of hypercups? For example, what if some of the hypercups were $2-$hypercups and some of the hypercups were $3-$hypercups? We do not know the answers, which means you could prove an original theorem!

\bibliographystyle{vancouver}
\bibliography{Cups.bib}

\end{document}